\documentclass[12pt,oneside,a4paper]{article}
\usepackage{amsmath,amsthm, amssymb}
\usepackage{hyperref}
\usepackage{cleveref}

\usepackage[left=3.0cm,right=3.0cm,top=2.0cm,bottom=2.0cm]{geometry}
\usepackage[utf8]{inputenc}

\newtheorem{definition}{Definition}

\newtheorem{theorem}[definition]{Theorem}
\newtheorem{theoremA}{Theorem}

\newtheorem{theoremB}{Theorem}

\newtheorem{corollary}[definition]{Corollary}
\newtheorem{conjecture}[definition]{Conjecture}

\theoremstyle{remark}
\newtheorem{remark}[definition]{Remark}

\begin{document}
\title{Sphere Theorems with and without Smoothing}
\author{Jialong Deng}
\date{}
\newcommand{\Addresses}{{
  \bigskip
  \footnotesize
  \textsc{Mathematisches Institut, Georg-August-Universit{\"a}t, G{\"o}ttingen, Germany}\par\nopagebreak
  \textit{Current address}: \textsc{Yau Mathematical Sciences Center, Tsinghua University, Beijing, China}\par\nopagebreak
  \textit{E-mail address}: \texttt{jialong.deng@mathematik.uni-goettingen.de}}}
\maketitle
\begin{abstract}
We show two sphere theorems for the Riemannian manifolds with scalar curvature bounded below and  the non-collapsed  $\mathrm{RCD}(n-1,n)$ spaces with mean distance close to $\frac{\pi}{2}$.
\end{abstract}

\section{Introduction}

Beginning with the Gauss-Bonnet theorem, sphere theorems show how  geometry can be used to decide the topology of a manifold. By asking the 1/4-pinching question, which is  now  a theorem, Hopf opened a door to the study of  sphere theorems.  More previous results and history can be found in the survey papers \cite{MR1452866}, \cite{MR2537082},\cite{MR2738904} etc.

 Continuing the paradigm, we will add two kinds of sphere theorems to enrich the subject. Details about the conditions  will be given later.

\begin{theoremA}\label{theorem A}

Let $(M,g)$ be an orientable closed  Riemannian $n$-manifold with scalar curvature$\geq n(n-1)$.  Suppose that  there exists a $(1, \wedge^1)$-contracting map $f: M \to S^n$ of non-zero degree and that the map $f$ is harmonic with condition $C\leq 0$, then $f$ is an isometric map.

\end{theoremA}

\begin{remark}
The idea of the proof is based on Chern-Goldberg's argument.
Their argument originated in  S.S. Chern's results that generalize Schwartz lemma to Hermitian manifolds in \cite{MR0234397} and was used to form the Riemannian version of Schwarz lemma in \cite{MR367860}.   In fact, the target of the map can be relaxed to general Einstein manifolds, which is shown by our proof.
\end{remark}

\begin{theoremB}\label{theorem B}

Let $(X,d, \mathcal{H}^n)$ be a  compact non-collapsed  $\mathrm{RCD}(n-1,n)$ space with full support and $\mathrm{md}(X)$ be the mean distance of $(X,d, \mathcal{H}^n)$, then
\begin{enumerate}
\item[I.]  $\mathrm{md}(X)\leq \frac{\pi}{2}$. 
\item[II.] $\mathrm{md}(X)= \frac{\pi}{2}$ if and only if $X$ is isometric to the standard round $n$-sphere $S^n$ and $\mathcal{H}^n=a\mathrm{dVol}$ for some $a>0$, where $\mathrm{dVol}$ is the volume form of $S^n$. 
\item[III.]  there is an $\epsilon(n)>0$ such that $\mathrm{md}(X)\geq \frac{\pi}{2}-\epsilon(n)$ implies that $X$ is homeomorphic to  $S^n$.
\end{enumerate}
\end{theoremB}

\begin{remark}
Mean distance is an old and well-studied metric invariant in graph theory, see \cite{MR485476}.  Note that mean distance is also called distance covariance in probability, see \cite{MR3127883}. 
\end{remark}

\begin{remark}
We will use Ketterer's maximal diameter theorem \cite[Theorem~1.4]{MR3333056} and Honda-Mondello's topological sphere theorem \cite[Theorem~A]{zbMATH07417792} (for the metric measure space) to prove Theorem \ref{theorem B}.  We also define a whole family of variants of the concept of mean distance for metric measure spaces  and  prove  sphere theorems for all of them.

\end{remark}

The paper is organized as follows. In Section \ref{section 2},  we introduce the notions of harmonic map with condition C  and prove Theorem \ref{theorem A}. In Section \ref{section 3}, we collect facts about  $\mathrm{RCD}(n-1,n)$ spaces and show Theorem \ref{theorem B}. 

$\mathbf{Acknowledgment}$: The author thanks Thomas Schick for his reading the  draft and helpful suggestions. The author also appreciates the anonymous reviewer's constructive feedback. In particular,  Remarks \ref{Comment 1} and \ref{Comment 2} were suggested by the reviewer. This note is  part of the author's thesis in \cite{1}.  The funding came from a doctoral scholarship of  Mathematisches Institut, Georg-August-Universit{\"a}t, G{\"o}ttingen.

 Finished in May 2021, this note was dedicated to  the $110^{\mathrm{th}}$ anniversary of  the birth of Shiing-Shen Chern.

\section{Harmonic Maps with Condition C}\label{section 2}
Let $S^n$ or $(S^n,g_{st})$ be the standard round $n$-sphere, then one has many rigidity theorems about $S^n$.  We will focus on the rigidity theorem with scalar curvature bounded below.

\begin{definition}
Let $M$ and $N$ be two Riemannian manifolds of the same dimension $n$. A smooth map $f: M\to N$ is said to be $\epsilon$-contracting  if $\|f_*\nu\|\leq \epsilon \|\nu\|$ for all tangent vectors $\nu$ on $M$.

A smooth map $f:M\to N$ is said to be $(\epsilon, \wedge^k)$-contracting if $\|f^*\varphi\|\leq \epsilon\|\varphi\|$, for all $k$-forms $\varphi \in \wedge^k(N)$ and $1\leq k\leq n$.
\end{definition}

Notice that $1$-contracting and volume contracting with contraction constant $1$ means $(1, \wedge^1)$- and $(1, \wedge^n)$-contracting, respectively.  If $f$ is $(1, \wedge^p)$-contracting, then $f$ is $(1, \wedge^{p+k})$-contracting for $k\geq 1$.

\begin{theorem}[Llarull \cite{MR1338312}]
For all $n$, $k$, $D$ with $n\geq k\geq 3$ and $D>0$,  there exists a Riemannian $n$-manifold $M^n$ with scalar curvature$\geq D$ and a $(1, \wedge^k)$-contracting map $f:M^n\to S^n$ of degree $1$. 
\end{theorem}  

In particular, a volume decreasing and non-zero degree map for $M^n$ (with scalar  curvature bounded below) into $S^n$ generally will  not be an isometric map.

By confirming and generalizing Gromov's conjecture \cite{MR859578}, Llarull showed the following rigidity theorem using index theoretical method.

\begin{theorem}[Llarull \cite{MR1600027}]
Let $(M^n,g)$ be a closed, connected, Riemannian spin $n$-manifold with scalar curvature$\geq n(n-1)$. Assume that there exists a $1$-contracting map or $(1, \wedge^2)$-contracting map $f: M^n \to S^n$ of non-zero degree, then $f$ is an isometric map.
\end{theorem}

Note that the spin condition on $M^n$ can be relaxed to require that the map $f$ is the spin map.  $S^n$ can be replaced by the Riemannian manifold, which  satisfies a non-negative curvature operator and certain harmonic spinor conditions \cite[Lemma~1.1]{MR1877585}. Llarull's rigidity theorem  can also  be generalized to the weighted rigidity theorem on weighted Riemannian manifold with positive weighted scalar curvature in \cite[Proposition~4.21]{MR4210894}.

\begin{remark}
It is natural to appeal to use  Schoen-Yau's minimal surface method \cite{MR541332}  to give another proof of Llarull's rigidity theorem.  Since Gromov-Lawson's index theoretical approach \cite{MR720933} and Schoen-Yau's minimal surface method are  two of the fundamental methods of studying  scalar curvature, one can  try to apply one approach to give a  proof of results that has been  showed by the other.

\end{remark}

Here we will use  harmonic maps to approach the rigidity problem on scalar curvature. Furthermore, our argument can be applied to the rigidity question of the arbitrary  Einstein manifolds.

We follow the setup and notations in \cite{MR404698}. Let $M$ and $N$ be Riemannian $n$-manifolds.  Let $ds^2_M$ and $ds^2_N$ be the Riemannian metrics of $M$ and $N$, respectively. Then we can write, locally,
\begin{equation*}
ds^2_M=\omega^2_1+ \cdots + \omega^2_n, \quad   ds^2_N= \omega^{*2}_1+ \cdots +\omega^{*2}_n,
\end{equation*}
where $\omega_i$ $(1\leq i\leq n)$ and $\omega^*_a$ $(1\leq a \leq n)$ are linear differential forms in $M$ and $N$, respectively. The structure equations in $M$ are 
\begin{align*}
&d\omega_i=\sum\limits_{j} \omega_j\wedge \omega_{ji}  \\
&d\omega_{ij}= \sum\limits_{k}\omega_{ik}\wedge \omega_{kj} - \frac{1}{2}\sum\limits_{k,l}R_{ijkl}\omega_k\wedge\omega_l.
\end{align*}
The Ricci tensor $R_{ij}$ is defined as
\begin{equation*}
R_{ij}=\sum\limits_{k}R_{ikjk}
\end{equation*}
and the scalar curvature is defined as
\begin{equation*}
R=\sum\limits_{i}R_{ii}.
\end{equation*}
Similar equations are valid in $N$ and  the corresponding quantities  are denoted in the same notation with asterisks.

Let $f$ be a $C^{\infty}$-mapping of $M$ into $N$ and 
\begin{equation*}
f^*\omega^*_a=\sum\limits_{i}A^a_i\omega_i.
\end{equation*}
Later, we will drop $f^*$ from such formulas when its presence is clear in the context.

Let $e_i$ (resp. $e^*_a$) be a frame that is dual to the coframe $\omega_i$ (resp. $\omega^*_a$), then we have 
\begin{equation*}
f_*e_i=\sum\limits_{a}A^a_ie^*_a.
\end{equation*}

The covariant differential of $A^a_i$ is defined by
\begin{equation}\label{1}
DA^a_i:= dA^a_i+ \sum\limits_{j}A^a_j\omega_{ji}+\sum\limits_{b}A^b_i\omega^*_{ba}:= \sum\limits_jA^a_{ij}\omega_j  
\end{equation}
with
\begin{equation}\label{1.5}
 A^a_{ij}=A^a_{ji}
\end{equation}

The following geometrical interpretation of $A^a_{ij}$ was given by Chern-Goldberg \cite[P.136]{MR367860}: Let $x\in M$ and let $T_x$ and $T_{f(x)}$ be the tangent spaces at $x$ and $f(x)$, respectively. The mapping 
\begin{equation*}
f_{**}: T_x\to T_{f(x)}
\end{equation*}
defined by 
\begin{equation*}
f_{**}(\nu)=\sum\limits_{a,i,j}A^a_{ij}\lambda_i\lambda_je^*_a, \quad \nu=\sum\limits_{i}\lambda_ie_i
\end{equation*}
is quadratic.  The mapping has the property that if $\nu$ is a unit vector, $f_{**}(\nu)$ is the acceleration vector of $f(\gamma)$ at $f(x)$, where $\gamma$ is the geodesic tangent to $\nu$ at $x$. 

\begin{definition}[Harmonic maps]
The mapping $f$ is called harmonic if 
\begin{equation*}
\sum\limits_{i}A^a_{ii}=0.
\end{equation*}
\end{definition}

Notice that the tensor field with the components $\sum\limits_{i}A^a_{ii}$ is the tension vector field of Eells-Sampson \cite{MR164306}.  Eells and Sampson  proved that the mapping $f$ is harmonic (in Eells-Sampson's sense) if and only if $\sum\limits_{i}A^a_{ii}=0$.

Differentiating the equation \eqref{1} and using the structure equations in $M$ and $N$, we get 
\begin{equation}\label{2}
\sum\limits_{j}DA^a_{ij}\wedge\omega_j=-\frac{1}{2}\sum\limits_{j,k,l}A^a_lR_{jikl}\omega_k\wedge\omega_l 
- \frac{1}{2}\sum\limits_{b,c,d}A^b_iR^*_{bacd}\omega^*_c\wedge\omega^*_d,
\end{equation}
where 
\begin{equation*}
DA^a_{ij}:= dA^a_{ij}+\sum_bA^b_{ij}\omega^*_{ba}+\sum\limits_kA^a_{kj}\omega_{ki}+ \sum\limits_{k}A^a_{ik}\omega_{kj}:= \sum\limits_kA^a_{ijk}\omega_k.
\end{equation*}

Thus, we have
\begin{equation}\label{3}
A^a_{ijk}-A^a_{ikj}=-\sum\limits_{l}A^a_lR_{likj}-\sum\limits_{b,c,d}A^b_iA^c_kA^d_jR^*_{bacd}.
\end{equation}

According to \eqref{1.5} and \eqref{3}, we can calculate the Laplacian
\begin{equation}\label{3.5}
\Delta A^a_i:=\sum\limits_kA^a_{ikk}=\sum\limits_kA^a_{kik}=\sum\limits_kA^a_{kki}+\sum\limits_lA^a_lR_{li}-\sum\limits_{b,c,d,k}R^*_{bacd}A^b_kA^c_kA^d_i.
\end{equation}

The ration  $ A:=\frac{f^*d\nu_N}{d\nu_M}$ of volume elements has the expression $A=\mathrm{det}(A^a_i)$. Let $(B^i_a)$ be the adjoint of $(A^a_j)$, i.e., $\sum\limits_aB^i_aA^a_j=\delta^i_jA$. Then, by \eqref{1},
\begin{equation*}
dA=\sum\limits_{i,a}B^i_adA^a_i=\sum\limits_{i,j,a}B^i_adA^a_{ij}\omega_j =: \sum\limits_j \mathbb{A}_j\omega_j
\end{equation*}

Let $V=A^2$, then we will compute the  Laplacian of $V$ and  use it to prove Theorem \ref{theorem A}.

First, let
\begin{equation}\label{4}
dV=\sum\limits_kV_k\omega_k,
\end{equation}
where
\begin{equation}\label{5}
V_k=2A\sum\limits_{i,a}B^i_aA^a_{ik}.
\end{equation}

Second, exterior differentiation of \eqref{4} gives
\begin{equation*}
\sum\limits_k(dV_k-\sum\limits_i V_i \omega_{ki})\wedge \omega_k =0.
\end{equation*}

Hence we can write
\begin{equation}\label{6}
dV_k-\sum\limits_i V_i\omega_{ki} =\sum\limits_jV_{kj}\omega_j,
\end{equation}
where
\begin{equation*}
V_{jk}=V_{kj}.
\end{equation*}

Then the Laplacian of $V$ is by definition equal to 
\begin{equation*}
\Delta V=\sum\limits_kV_{kk}.
\end{equation*}

Third, by differentiating \eqref{5}, using \eqref{6}, and simplifying, we get
\begin{equation*}
\frac{1}{2}V_{kj}=2\mathbb{A}_j\mathbb{A}_k-\sum\limits_{i,l,a,b}B_a^iB_b^lA_{lj}^aA_{ik}^b+A\sum\limits_{i,a}B_a^iA_{ikj}^a.
\end{equation*}

Using \eqref{3.5},  the  Laplacian of $V$ can be given by 
\begin{equation}\label{8}
\frac{1}{2}\Delta V= 2\sum\limits_j(\mathbb{A}_j)^2+V(R-\sum\limits_{b,c,j}R^*_{b,c}A^b_jA^c_j)-C+A\sum\limits_{a,i,j}B^i_aA^a_{jji},
\end{equation}
where $C$ is a scalar invariant of the mapping given by
\begin{equation}\label{c} 
C=\sum\limits_{\mathop{a,b}\limits_{i,j,k}}B^i_aB^k_bA^a_{kj}A^b_{ij}.
\end{equation}
If $f$ is harmonic, then the last term of \eqref{8} vanishes. That is, see \cite[P.141,~formula (57)]{MR367860},
\begin{equation}
\begin{aligned}
&\text{(for a harmonic map)}\\
\frac{1}{2}\Delta V= 2\sum\limits_j(\mathbb{A}_j)^2+V(R-\sum\limits_{b,c,j}R^*_{b,c}A^b_jA^c_j)-C.
\end{aligned}
\end{equation}

The geometric meaning of Condition C \eqref{c} was also given by Chern-Goldberg  \cite[P.142,~Remark]{MR367860} as follows:  the scalar C may be interpreted geometrically as a weighted measure of the deviation of the square length of the tensor $C_{ijk}$ from the square length of its symmetric part, where $C_{ijk}$ is the pullback of $A^a_{ij}$ under $f$.

\begin{proof}[Proof of Theorem \ref{theorem A}]
From the fact that the round sphere $S^n$ is an Einstein manifold, we get
\begin{equation*}
\sum\limits_{b,c,j}R^*_{b,c}A^b_jA^c_j=\frac{R^*}{n}\sum\limits_{a,i}(A^a_i)^2
\end{equation*}

Since $M$ is compact and $f$ is non-zero degree,  $V$ does attain its maximum at the point $x$ in $M$.  Then  $V(x)>0$ and $\Delta V(x)\leq 0$. Notice that $V(x)$ is independent of the choice of the frame and coframe.  At the point $x$, we choose a local $g$-orthonormal frame $e_1, \dots, e_n$ on $T_xM$ and a local $g_{st}$-orthonormal frame $e^*_1,\dots, e^*_n$ on $T_{f(x)}S^n$, such that there exists $\lambda_1\geq \lambda_2\geq\cdots \geq \lambda_n>0$, with $f_*e_i=\lambda_ie^*_i$. This can be done by diagonalizing $f^*g_{st}$ with respect to the metric $g$. As $f$ is $(1, \wedge^1)$-contracting, we have $\lambda_i\leq 1$  for all $1\leq i\leq n$. That means 
\begin{equation}\label{9}
\sum\limits_{b,c,j}R^*_{b,c}A^b_jA^c_j=\frac{R^*}{n}\sum\limits_{a,i}(A^a_i)^2\leq R^*.
\end{equation}
 
On account of $R^*=n(n-1)$, $R\geq n(n-1)$, and $C\leq 0$, we have
\begin{equation*}
\frac{1}{2}\Delta V(x)\geq V(x)(R-n(n-1))\geq 0.
\end{equation*}
Then by combining $V(x)>0$ and $\Delta V(x)\leq 0$, we get $R=n(n-1)$ and $\lambda_i=1$ for all $1\leq i\leq n$. That means $f$ is an isometric map.
\end{proof}

In fact, the target of the map can be relaxed to an arbitrary Einstein manifold as \eqref{9} shows.

\begin{corollary}
Let $(N,g)$ be an orientable closed (Riemannian) Einstein $n$-manifold with scalar curvature $R^*$ and $(M,g)$ be an orientable closed Riemannian  $n$-manifold with scalar curvature $R$. 

 Suppose  $R\geq R^*$,  there exists a $(1, \wedge^1)$-contracting map $f: M \to N$ of non-zero degree, and the map $f$ is harmonic with condition $C\leq 0$, then $f$ is a locally isometric map. If the fundamental group of $N$ is trivial, then $f$ is an isometric map.
\end{corollary}
\begin{proof}
As $N$ is an Einstein manifold,  then the inequality \eqref{9} and $R\geq R^*$ implies  $f$ is a locally isometric map.
\end{proof}

It is not clear whether the method of the poof in this section  can be generalized to weighted Riemannian manifold with positive weighted scalar curvature in \cite{MR4210894}.

\begin{conjecture}\label{conjecture}
Let $(M,g)$ be an orientable closed  Riemannian $n$-manifold with scalar curvature$\geq n(n-1)$.  Suppose that  there exists a $(1, \wedge^1)$-contracting map $f: M \to S^n$ of non-zero degree, then $f$ is an isometric map.
\end{conjecture}

\begin{remark}
If  conjecture \ref{conjecture} is confirmed, then Theorem B in \cite{MR4291609} holds for all dimensions. And then we can get a new obstruction of the existence of a Riemannian metric with positive scalar curvature on an arbitrary  closed smooth manifold.

\end{remark}

\section{Non-collapsed RCD Spaces and Mean Distance} \label{section 3}

The framework of Riemannian manifolds with Ricci curvature bounded below is   generalized to the metric measure space $(X,d,\mu)$, which  satisfies Riemannian curvature-dimension condition. And  the metric measure spaces are called $\mathrm{RCD}(K,N)$ spaces.

\begin{definition}[Non-collapsed   $\mathrm{RCD}(K,n)$ spaces ]
Let $K\in \mathbb{R}$, $n\in \mathbb{N}$ and $n\geq 1$.  The metric measure space $(X,d,\mu)$ is called  a  non-collapsed  $\mathrm{RCD}(K,n)$ space, if 
\begin{enumerate}
\item[1.]  $(X,d,\mu)$ is an $\mathrm{RCD}(K,n)$ space.
\item[2.] [Non-collapsed condition] $\mu=\mathcal{H}^n$, where $\mathcal{H}^n$ is the $n$-dimensional Hausdorff measure with respect to the metric $d$.
\end{enumerate}

\end{definition}

The  concept of non-collapsed   $\mathrm{RCD}(K,N)$ space  was defined by De Philippis and Gigli \cite[Definition~1.1]{MR3852263}.   De Philippis-Gigli showed that $N$ must be an integer.  The proof of our theorem  below is mainly based on  De Philippis and Gigli's results.

\begin{definition}[Mean distance]
Let $(X,d, \mu)$ be a compact metric measure space with full support, then the mean distance of $X$ is defined as
\begin{equation*}
\mathrm{md}(X):=\int\limits_{X\times X} \frac{ d(\cdot, \cdot)}{\mu(X)\times \mu(X)} \mu\otimes \mu  
\end{equation*}
\end{definition}

 Using Cheng's maximal diameter theorem and Cheeger-Colding's differentiable  sphere theorem \cite{MR1484888},  Kokkendorff \cite[Theorem~4]{MR2474426} showed the following theorem.
\begin{theorem}[Kokkendorff]
Let $(M, g, d\mathrm{Vol}_g)$ be a Riemannian $n$-manifold $(n\geq 2)$ with Ricci curvature$\geq n-1$, where $ d\mathrm{Vol}_g$ is the Riemannian volume,  then
\begin{enumerate}
\item[1.] $\mathrm{md}(M)\leq \frac{\pi}{2}$ 
\item[2.] $\mathrm{md}(M)= \frac{\pi}{2}$  if and only if $M$ is isometric to the standard round $n$-sphere $(S^n,  d_{S^n})$.
\item[3.] there is an $\epsilon'(n)>0$ such that $\mathrm{md}(M)\geq \frac{\pi}{2}-\epsilon'(n)$ implies that $M$ is diffeomorphic to  $S^n$.
\end{enumerate} 
 
\end{theorem}

We  generalize Kokkendorff's theorem to  compact non-collapsed  $\mathrm{RCD}(n-1,n)$ spaces.

The strategy of the proof of Theorem \ref{theorem B} is the same as Kokkendorff's proof, but we need to replace Cheng's maximal diameter theorem and Cheeger-Colding's differentiable  sphere theorem (for the Riemannian manifold) with Ketterer's maximal diameter theorem \cite[Theorem~1.4]{MR3333056} and Honda-Mondello's topological sphere theorem \cite[Theorem~A]{zbMATH07417792} (for  metric measure spaces). 

 The radius $\mathrm{rad}(X,d)$ of a metric space $(X,d)$ is defined as
  \begin{equation*}
 \mathrm{rad}(X,d):=\inf\limits_{x\in X}\sup\limits_{y\in X}d(x,y).
  \end{equation*}
    
  Before giving the proof of Theorem \ref{theorem B}, let us collect the relevant  properties of a non-collapsed  $\mathrm{RCD}(n-1,n)$ space with full support so that we can use them in the proof. 
  \begin{enumerate}
  \item[(1).]  The diameter $\mathrm{diam}(X,d)$ is at most $\pi$ (in particular $\mathrm{rad}(X,d)\leq \pi$)  \cite[Theorem~4.8]{MR3333056}.
  \item[(2).]  $(X,d, \mathcal{H}^n)$ satisfies the generalized Bishop-Gromov inequality:
  \begin{equation*}
  \frac{ \mathcal{H}^n(B_r(x))}{ \mathcal{H}^n(B_R(x))}\geq \frac{\mathrm{Vol}_{S^n}(B_r)}{\mathrm{Vol}_{S^n}(B_R)}
  \end{equation*}
  for any $x\in X$ and  $0\leq r \leq R \leq \pi$, where $B_r(x)$ is the closed $r$-ball with center $x$ in $X$,  $B_s$ is a  closed $s$-ball in the sphere $S^n$, and $\mathrm{Vol}$ is the volume on the $S^n$. Equivalently
  \begin{equation*}
  \frac{ \mathcal{H}^n(B_r(x))}{ \mathcal{H}^n(B_R(x))}\geq \frac{\int_0^r[\sin(t)]^{n-1}dt}{\int_0^R[\sin(t)]^{n-1}dt}
  \end{equation*}
   for any $x\in X$ and  $0\leq r \leq R \leq \pi$   \cite{MR3852263}.

  \item[(3).] [Ketterer's  maximal diameter theorem ] If $\mathrm{rad}(X,d)=\pi$, then $(X,d)$ is isometric to $(S^n, d_{S^n})$ \cite[Theorem~1.4]{MR3333056}.
  \item[(4).] [Honda-Mondello's topological sphere theorem] For all $n\in \mathbb{N}_{\geq 2}$,  there exists a positive constant $\epsilon_1(n)>0$ such that if a compact metric space $(X,d)$ satisfies that $\mathrm{rad}(X,d)\geq \pi- \epsilon_1(n)$, and that $(X,d, \mu)$ is an $\mathrm{RCD}(n-1,n)$ space for some Borel measure $\mu$ on $X$ with full support, then $X$ is homeomorphic to the $n$-dimensional sphere \cite[Theorem~A]{zbMATH07417792}.
  
  \end{enumerate}

\begin{proof}[Proof of Theorem \ref{theorem B}]
Since the diameter of a non-collapsed  $\mathrm{RCD}(n-1,n)$ space is at most $\pi$, we can apply the argument in the proof of Theorem 1.3 in \cite{MR4186467} to prove  I.

 Let $\mathrm{D}(x):=\sup\{d(x,y):y\in X\}\leq \pi$ for a fixed $x\in X$,  $h(y):=d(x,y)$, and $\mu_0:=\frac{\mathcal{H}^n}{\mathcal{H}^n(X)}$, then we have
 \begin{align}
\mathrm{md}(x)&:= \int_Xh(y)\mu_0(dy)=\int_0^\infty\mu_0(\{h\geq s\})ds \\
      &=\int_0^{\mathrm{D}(x)}[1-\mu_0(h^{-1}([0,s))]ds=\int_0^{\mathrm{D}(x)}[1-\mu_0(B_s(x))]ds.\label{d}
 \end{align}

 We have the generalized Bishop-Gromov inequality for   $\mathrm{RCD}(n-1,n)$ space, i.e.,
\begin{equation}\label{Bishop-Gromov}
\mu_0(B_s(x))\geq \frac{\mathrm{Vol}(B_s)}{\mathrm{Vol}(S^n)}
\end{equation}
for $x\in X$ and  $0\leq s\leq \pi$. Hence, we get 
\begin{equation}\label{r}
 \mathrm{md}(x)=\int_Xh(y)\mu_0(dy)\leq \pi -\int_0^\pi  \frac{\mathrm{Vol}(B_s)}{\mathrm{Vol}(S^n)}ds.
\end{equation} 
The right hand side of the inequality \eqref{r} coincides with $\mathrm{md}(p)$ for any $p\in S^n$ by virtue of the formula in \eqref{d}.  As $\mathrm{md}(S^n)=\frac{\pi}{2}$ was showed by Kokkendorff above, we have $\mathrm{md}(x) \leq \frac{\pi}{2}$. That means $\mathrm{md}(X)\leq \frac{\pi}{2}$. 

 If we have equality $\mathrm{md}(X)=\frac{\pi}{2}$, then Bishop-Gromov inequality \eqref{Bishop-Gromov} must be equality for all $0\leq s\leq \pi$. It implies that for $\mu_0$-a.e. point $x$, there must exist a point $x'$ with $d(x,x')=\pi$. And this must hold even for every $x\in X$, because $X$ is compact and $\mu_0$ is full support.  Therefore, $\mathrm{rad}(X)=\pi$ and then  $(X,d)$ is isometric to $(S^n, d_{S^n})$ by Ketterer's  maximal diameter theorem.
 
 We will prove III. by showing that $\mathrm{md}(X)$  close to $\frac{\pi}{2}$ implies that $\mathrm{rad}(X,d)$ is closed to $\pi$ and by applying Honda-Mondello's topological sphere theorem. 
 
  We will prove the claim by contradiction.  Suppose that there exists $x_0\in X$ such that $D(x_0)< \pi-\epsilon_1(n)$ and that  the $x_0$ realizes $\mathrm{rad}(X,d)$. Here we take $\epsilon_1(n)$ from Honda-Mondello's topological sphere theorem. Then $\mu_0(B_s(x_0))$ will achieve its maximum value $1$ for $ \pi-\epsilon_1(n)$. Hence,  we get an estimate
 \begin{align*}
\mathrm{md}(x_0)&=\int_0^{\mathrm{D}(x)}[1-\mu_0(B_s(x))]ds\leq \int_0^{\pi-\epsilon_1(n)}[1-\frac{\mathrm{Vol}(B_s)}{\mathrm{Vol}(S^n)}]ds\\
&= \int_0^{\pi}[1-\frac{\mathrm{Vol}(B_s)}{\mathrm{Vol}(S^n)}]ds- \int_{\pi-\epsilon_1(n)}^\pi[1-\frac{\mathrm{Vol}(B_s)}{\mathrm{Vol}(S^n)}]ds =:\frac{\pi}{2}- \delta(\epsilon_1(n),n).
\end{align*}   

 Since $\mathrm{D}$ is $1$-Lipschitz, we have
 \begin{equation*}
 \mathrm{D}(y)< \pi- \frac{\epsilon_1(n)}{2}
 \end{equation*}
 for any $y\in B_{ \frac{\epsilon_1(n)}{2}}(x_0)$. This implies that we have the estimate
 \begin{equation*}
 \mathrm{md}(y)<\frac{\pi}{2}- \delta(\frac{\epsilon_1(n)}{2},n)
 \end{equation*}
for any $y\in B_{\frac{\epsilon_1(n)}{2}}(x_0)$. Finally, we have
\begin{align*}
\mathrm{md}(X) &=\int_X \mathrm{md}(x)\mu_0(dx)=\int_{X \setminus B_{ \frac{\epsilon_1(n)}{2}}(x_0)}\mathrm{md}(x)\mu_0(dx)+ \int_{ B_{ \frac{\epsilon_1(n)}{2}}(x_0)}\mathrm{md}(x)\mu_0(dx)\\
               &< \frac{\pi}{2}- \mu_0(B_{\frac{\epsilon_1(n)}{2}}(x_0))\delta(\frac{\epsilon_1(n)}{2},n)\leq \frac{\pi}{2}-\frac{\mathrm{Vol}(B_\frac{\epsilon_1(n)}{2})}{\mathrm{Vol}(S^n)}\delta(\frac{\epsilon_1(n)}{2},n).
\end{align*}
Now we take
\begin{equation*}
\epsilon(n)=\frac{\mathrm{Vol}(B_\frac{\epsilon_1(n)}{2})}{\mathrm{Vol}(S^n)}\delta(\frac{\epsilon_1(n)}{2},n)>0.
\end{equation*}
 Then $\mathrm{md}(X)\geq \frac{\pi}{2}-\epsilon(n)$ implies $\mathrm{rad}(X,d)\geq \pi- \epsilon_1(n)$. Otherwise, it would be a contradiction.  Then   $X$ is homeomorphic to the $n$-dimensional sphere by Honda-Mondello's topological sphere theorem.
\end{proof}

\begin{remark}\label{Comment 1}
The definition of metric measure spaces with the measure contraction property $\mathrm{MCP}(K, n)$ was given independently by Ohta \cite{MR2341840} and Sturm \cite{MR2237207}. On general metric measure spaces, the two definitions slightly differ, but on essentially non-branching spaces they coincide \cite[Appendix~A]{MR3691502}.  $\mathrm{MCP}(K, n)$ spaces are another kind of  generalization of Riemannian manifolds with Ricci curvature lower bound, but  a measure contraction property  is weaker than the usual curvature dimension conditions \cite[Remark~5.6]{MR2237207}.

 However, $\mathrm{MCP}(K, n)$ spaces also satisfy generalized Bishop–Gromov  inequality and the generalized Bonnet–Myers theorem  \cite[Remark~5.2]{MR2237207} \cite[Theorem~4.3;~Theorem~5.1]{MR2341840}. Therefore, Theorem \ref{theorem B}.  I. also holds for a compact non-collapsed  $\mathrm{MCP}(n-1, n)$ space with full support.
 
Ketterer and Rajala  show that the non-branching assumption is essential in   Ohta's topological rigidity \cite[Theorem~5.5]{MR2351619} in \cite[Theorem~2]{MR3336992}. It is not clear (for the author)  whether the  maximal diameter (rigidity) theorem holds for  a compact non-collapsed  and  non-branching $\mathrm{MCP}(n-1, n)$ space with full support.
\end{remark}

\begin{remark}\label{Comment 2}
 A  notion for lower bounds of Ricci curvature on Alexandrov spaces was introduced by Zhang-Zhu in \cite{MR2747437} \cite{MR2778704}. Bisho-Gromov inequality and maximal diameter theorem hold for Zhang-Zhu's spaces \cite[Corollary~5.1]{MR2747437}. Thus, one can show that Theorem \ref{theorem B}. I. and II.  hold for  an $n$-dimensional Alexandrov space without boundary, and  with full support and   Ricci curvature$\geq n-1$ in  Zhang-Zhu's sense. 
\end{remark}

Inspired by the concept of  mean distance, one can define similar metric invariants.  For  a compact metric measure space $(X,d, \mu)$  with full support,  given  a continuous function $\alpha:[0,\infty)\to \mathbb{R}$,  we can define the metric invariant $\mathrm{M_{\alpha}}(X)$ as 
\begin{equation*}
\mathrm{M_{\alpha}}(X):=\int\limits_{X\times X} \frac{\alpha( d(\cdot, \cdot))}{\mu(X)\times \mu(X)} \mu\otimes \mu  
\end{equation*}
And then one can try to generalize Theorem \ref{theorem B} to  $\mathrm{M_{\alpha}}(X)$ case. 

In particular,  Erbar and Sturm \cite{MR4186467} defines 
\begin{equation*}
\mathrm{M_{f}}(X):=\int\limits_{X\times X} \frac{f( d(\cdot, \cdot))}{\mu(X)\times \mu(X)} \mu\otimes \mu,  
\end{equation*}

\begin{equation*}
\mathrm{M_{f, n}^*}:= \frac{\int_0^\pi f(r)[\sin(r)]^{n-1}dr }{\int_0^\pi [\sin(r)]^{n-1}dr}, 
\end{equation*}
where $f:[0,\pi]\to \mathbb{R}$ is a continuous and strictly increasing function.

\begin{theorem}[Erbar-Sturm]
Let $(X,d,\mu)$ be a compact non-collapsed $\mathrm{RCD}(n-1,n)$ space with $n\geq 1$, then $\mathrm{M_{f}}(X)\leq \mathrm{M_{f, n}^*}$ and $\mathrm{M_{f}}(X)= \mathrm{M_{f, n}^*}$ if and only if $X$ is isometric to the standard round $n$-sphere $S^n$. 
\end{theorem}

Erbar-Sturm also noticed that an analogous statement (with $\mathrm{M_{f}}(X)\geq  \mathrm{M_{f, n}^*}$ in place of  $\mathrm{M_{f}}(X)\leq \mathrm{M_{f, n}^*}$) holds for strictly decreasing $f$. Without loss of  generality, let us assume that $f$ is  continuous and strictly increasing. Then we have the following corollary by combining Erbar-Sturm's theorem and the proof of Theorem \ref{theorem B}.

\begin{corollary}\label{sphere}
Let $(X,d, \mathcal{H}^n)$ be a  compact non-collapsed  $\mathrm{RCD}(n-1,n)$ space with full support, then there is an $\epsilon_2(f,n)>0$ such that $\mathrm{M_{f}}(X)\geq \mathrm{M_{f, n}^*}-\epsilon_2(f,n)$ implies that $X$ is homeomorphic to  $S^n$.
\end{corollary}

\begin{proof}
We only need to show that  $\mathrm{M_{f}}(X)$ close to $\mathrm{M_{f, n}^*}$ implies that $\mathrm{rad}(X,d)$ is closed to $\pi$.  We will prove the claim by contradiction. 

 Suppose that there exists $x_0\in X$ such that $D(x_0)< \pi-\epsilon_1(n)$ and the $x_0$ realizes $\mathrm{rad}(X,d)$. Here we take $\epsilon_1(n)$ from Honda-Mondello's topological sphere theorem. Let $g(y):=f(d(x_0,y))$,  $\mu_0:=\frac{\mathcal{H}^n}{\mathcal{H}^n(X)}$, and $\mathrm{M_{f}}(x_0):=\int_X g(y)\mu_0(dy)$,  then $\mu_0(B_s(x_0))$ will achieve its maximum value $1$ for $ \pi-\epsilon_1(n)$.
 
  Hence,  we get an estimate
  \begin{equation*}
  \mathrm{M_{f}}(x_0)<\mathrm{M_{f, n}^*} - \int_{f(\pi-\epsilon_1(n))}^{f(\pi)}[1- \frac{\int_0^{f^{-1}(s)}[\sin(t)]^{n-1}dt}{\int_0^{\pi}[\sin(t)]^{n-1}dt}]ds:=\mathrm{M_{f, n}^*}-  \delta_1(f,\epsilon_1(n),n).
  \end{equation*}
  
 Since $\mathrm{D}$ is $1$-Lipschitz and $f$ is continuous and strictly increasing, we have
 \begin{equation*}
 f(\mathrm{D}(y))< f(\pi- \frac{\epsilon_1(n)}{2})
 \end{equation*}
 for any $y\in B_{ \frac{\epsilon_1(n)}{2}}(x_0)$. It implies that we have the estimate
 \begin{equation*}
  \mathrm{M_{f}}(y)<\mathrm{M_{f, n}^*}-  \delta_1(f, \frac{\epsilon_1(n)}{2}),n)
 \end{equation*}
for any $y\in B_{\frac{\epsilon_1(n)}{2}}(x_0)$. Finally, we have
\begin{align*}
  \mathrm{M_{f}}(X) &=\int_X   \mathrm{M_{f}}(x)\mu_0(dx)=\int_{X \setminus B_{ \frac{\epsilon_1(n)}{2}}(x_0)}  \mathrm{M_{f}}(x)\mu_0(dx)+ \int_{ B_{ \frac{\epsilon_1(n)}{2}}(x_0)}  \mathrm{M_{f}}(x)\mu_0(dx)\\
               &< \mathrm{M_{f, n}^*}- \mu_0(B_{\frac{\epsilon_1(n)}{2}}(x_0))\delta_1(f,\frac{\epsilon_1(n)}{2},n)\leq  \mathrm{M_{f, n}^*}-\frac{\mathrm{Vol}(B_\frac{\epsilon_1(n)}{2})}{\mathrm{Vol}(S^n)}\delta_1(f, \frac{\epsilon_1(n)}{2},n).
\end{align*}

Now we take
\begin{equation*}
 \epsilon_2(f,n)=\frac{\mathrm{Vol}(B_\frac{\epsilon_1(n)}{2})}{\mathrm{Vol}(S^n)}\delta_1(f, \frac{\epsilon_1(n)}{2},n))>0.
\end{equation*}
 Then $\mathrm{M_{f}}(X)\geq \mathrm{M_{f, n}^*}-\epsilon(n)$ implies $\mathrm{rad}(X,d)\geq \pi- \epsilon_1(n)$. Otherwise, it would be a contradiction.  Then   $X$ is homeomorphic to the $n$-dimensional sphere by Honda-Mondello's topological sphere theorem.

\end{proof}

\begin{remark}
The result of Corollary \ref{sphere} is even new (for $f\neq \mathrm{id}$) for a smooth Riemannian $n$-manifold with Ricci curvature bounded below by $(n-1)$ and $X$ is  diffeomorphic to  $S^n$ by Cheeger-Colding's  differentiable  sphere theorem.

\end{remark}

Gromov defines the observable diameter $\mathrm{ObsDiam}(X;-\kappa)$ for metric measure spaces $(X,d,\mu)$ in \cite{MR1699320} and shows the following theorem. The detailed proof can be found in \cite[Theorem~2.29]{MR3445278}.

\begin{theorem}[Gromov]
Let $X$ be a closed $n$-dimensional Riemannian manifold with Ricci curvature$\geq (n-1)$. Then, for any $\kappa$ with $0<\kappa\leq 1$, we have 
\begin{equation*}
 \mathrm{ObsDiam}(X;-\kappa)\leq  \mathrm{ObsDiam}(S^n;-\kappa)=\pi-2v^{-1}(\frac{\kappa}{2})
\end{equation*}
where 
\begin{equation*}
v(r):=\frac{\int_0^r[\sin(t)]^{n-1}dt}{\int_0^{\pi}[\sin(t)]^{n-1}dt}.
\end{equation*}

\end{theorem} 

It is not clear whether Gromov's theorem can be generalized to non-collapse $\mathrm{RCD}(n-1,n)$ spaces. If it can, can we show  the sphere theorem about it?
\\

The note has no associated data.

\bibliographystyle{alpha}
\bibliography{reference}
\Addresses

\end{document}